\documentclass[10pt]{amsart}
\usepackage{palatino}
\usepackage{amsfonts}
\usepackage{amssymb}
\usepackage{amscd}
\usepackage{hyperref}
\newtheorem{thm}{Theorem}[section]
\newtheorem{cor}[thm]{Corollary}
\newtheorem{conjecture}[thm]{Conjecture}
\newtheorem{lemma}[thm]{Lemma}
\newtheorem{prop}[thm]{Proposition}
\theoremstyle{definition}
\newtheorem{defn}[thm]{Definition}
\newtheorem{ex}[thm]{Example}
\theoremstyle{rem}
\newtheorem{rem}[thm]{Remark}
\numberwithin{equation}{subsection}

\begin{document}

\title[The tautological ring of $M_{1,n}^{ct}$]{The tautological ring of $M_{1,n}^{ct}$}
\author[M. Tavakol]{Mehdi Tavakol}
   \address{Department of Mathematics,
              KTH, 100 44 Stockholm, Sweden}
   \email{tavakol@math.kth.se}

\maketitle
\centerline {\bf Introduction.}
Let $M_{g,n}^{ct}$ be the moduli space of stable $n$-pointed genus $g$ curves of compact type and denote by $R^*(M_{g,n}^{ct})$
its tautological ring. Here, we study this ring in genus one. It is known that the tautological ring $R^*(M_{1,n}^{ct})$ is additively generated by boundary cycles, 
and it is the subalgebra of the Chow ring $A^*(M_{1,n}^{ct})$ (taken with $\mathbb{Q}$-coefficients throughout) of $M_{1,n}^{ct}$ 
generated by divisor classes $D_I,$ for $I \subset \{1,...,n\}$ with $|I|>1.$ Recall that a boundary cycle of $M_{1,n}^{ct}$ parameterizes 
stable curves whose dual graphs are trees, and $D_I$ is associated to those with one edge, for which $I$ is the marking set on the genus zero component. 
We study this ring to understand the space of relations among the generators. In particular, we prove that the tautological ring is Gorenstein.

We begin this note by recalling the definitions and known facts about the tautological algebras as well as the conjectural structure of them.

In the second section we consider a fixed pointed elliptic curve $(C;O),$ and we describe the reduced fiber of the projection $M_{1,n}^{ct} \rightarrow M_{1,1}^{ct}$ over $[(C;O)] \in M_{1,1}^{ct},$ 
which is denoted by $\overline{U}_{n-1},$ as a sequence of blow-ups of $C^{n-1}.$ As a result, we get a map $$F:\overline{U}_{n-1} \rightarrow M_{1,n}^{ct}.$$ 
There is a description of the Chow ring $A^*(\overline{U}_{n-1})$ of $\overline{U}_{n-1}$ in the third section.

Then we define the tautological ring $R^*(C^n)$ of $C^n$ as a subring of its Chow ring $A^*(C^n).$ We give a description of the pairing 
$$R^d(C^n) \times R^{n-d}(C^n) \rightarrow \mathbb{Q}$$ for $0 \leq d \leq n.$ In particular, we will see that this pairing is perfect. 

The fifth section starts with the definition of the tautological ring $R^*(\overline{U}_{n-1})$ of $\overline{U}_{n-1}.$ It is defined to be the subalgebra 
of its Chow ring generated by the tautological classes in $R^*(C^{n-1})$ and the classes of proper transforms of the exceptional divisors introduced 
in the construction of $\overline{U}_{n-1}.$ The study of the pairing $$R^d(\overline{U}_{n-1}) \times R^{n-1-d}(\overline{U}_{n-1}) \rightarrow \mathbb{Q},$$
for $0 \leq d \leq n-1,$ shows that it is perfect as well.

In the last section we study the fibers of the map $F:\overline{U}_{n-1} \rightarrow M_{1,n}^{ct},$ and we will see that the images of the tautological classes in 
$M_{1,n}^{ct}$ under the induced pull-back $$F^*:A^*(M_{1,n}^{ct}) \rightarrow A^*(\overline{U}_{n-1})$$ are elements of the tautological ring 
$R^*(\overline{U}_{n-1})$ of $\overline{U}_{n-1}$ and hence, it induces a map $$R^*(M_{1,n}^{ct}) \rightarrow R^*(\overline{U}_{n-1}),$$ 
which is denoted by the same letter $F^*$, by abuse of notation. Then, we will see that $F^*$ induces an isomorphism between the tautological rings involved. 
This gives a description of the ring $R^*(M_{1,n}^{ct})$ in terms of the generators $D_I$'s and the space of relations. In particular, from the proven result for 
$R^*(\overline{U}_{n-1}),$ we conclude that $R^*(M_{1,n}^{ct})$ is a Gorenstein ring.

\vspace{+10pt}
\noindent{\bf Acknowledgments.}
The author wishes to thank Carel Faber for introducing this project and useful discussions. The comments of Eduard Looijenga and Rahul Pandharipande are appreciated as well. 

\section{\bf  Review of known facts and conjectures about the tautological ring $R^*(M_{g,n}^{ct})$}
\medskip

Let  $\overline{M}_{g,n}$ be the moduli space of stable curves of genus $g$ with $n$ marked points. In [FP3] the system of tautological rings is defined to 
be the set of smallest $\mathbb{Q}$-subalgebras of the Chow rings, $$R^*(\overline{M}_{g,n}) \subset A^*(\overline{M}_{g,n}),$$satisfying the following two properties:
 
\begin{itemize}

\item
The system is closed under push-forward via all maps forgetting markings: $$\pi_*:R^*(\overline{M}_{g,n}) \rightarrow R^*(\overline{M}_{g,n-1}).$$

\item
The system is closed under push-forward via all gluing maps: 
$$\iota_*: R^*(\overline{M}_{g_1,n_1 \cup \{*\}}) \otimes  R^*(\overline{M}_{g_2,n_2 \cup \{ \bullet \}})  \rightarrow R^*(\overline{M}_{g_1+g_2,n_1+n_2}),$$
$$\iota_*: R^*(\overline{M}_{g,n \cup \{ *,\bullet \}}) \rightarrow R^*(\overline{M}_{g+1,n}),$$ with attachments along the markings $*$ and $\bullet.$
\end{itemize}

The standard $\psi,\kappa$ and $\lambda$ classes in $A^*(\overline{M}_{g,n})$ all lie in the tautological ring (see below for definition). The quotient $R^*(M_{g,n})$ of the tautological ring is defined as the restriction to the open subset $M_{g,n}.$
In [F1] it was conjectured that the tautological ring $R^*(M_g)$ is a Gorenstein algebra with socle in degree $g-2.$ It was raised as a question in [HL] whether the tautological ring of $\overline{M}_{g,n}$ satisfy Poincare duality and has the 
Lefschetz property with respect to $\kappa_1$, which was known to be ample [C]. In [F2] the following conjecture about the tautological ring $R^*(\overline{M}_{g,n})$ is stated:

\begin{conjecture}\label{bar}
$R^*(\overline{M}_{g,n})$ is Gorenstein with socle in degree $3g-3+n.$ 
\end{conjecture}

We now define the moduli space $M_{g,n}^{ct}$ and its tautological ring. To every $n$-pointed stable curve $(C;x_1,...,x_n)$ there is an associated dual graph. 
Its vertices correspond to the irreducible components of $C$ and edges correspond to intersection of components. Note that self intersection is allowed.
The curve $C$ is of compact type if its dual graph is a tree, or equivalently, the Jacobian of $C$ is an abelian variety. The moduli space $M_{g,n}^{ct}$ 
parametrizes stable $n$-pointed curves of genus $g$ of compact type. One can also define $M_{g,n}^{ct}$ as the complement of the boundary divisor $\Delta_{irr}$ 
of irreducible singular curves and their degenerations.

The tautological ring, $R^*(M_{g,n}^{ct}) \subset A^*(M_{g,n}^{ct})$, for the moduli space $M_{g,n}^{ct}$, is defined to be the image of $R^*(\overline{M}_{g,n})$ via the natural map,
$$R^*(\overline{M}_{g,n}) \subset A^*(\overline{M}_{g,n}) \rightarrow A^*(M_{g,n}^{ct}).$$

The quotient ring $R^*(M_{g,n}^{ct})$ admits a canonical non-trivial linear evaluation $\epsilon$ to $\mathbb{Q}$ obtained by integration involving the $\lambda_g$ class, the Euler class of the Hodge bundle.  

Recall that the Hodge bundle $\mathbb{E}$ is the locally free $Q$-sheaf of rank $g$  defined by  $\mathbb{E}=\pi_* \omega,$ 
where $\pi:\overline{M}_{g,n+1} \rightarrow \overline{M}_{g,n}$ forgets the marking $n+1$ and $\omega$ denotes its relative dualizing sheaf.  
The fiber of $\mathbb{E}$ over a moduli point $[(C;x_1,...,x_n)]$ is the $g$-dimensional vector space $H^0(C,\omega_C).$ The class $\lambda_i$ is defined to be the $i^{th}$ Chern class $c_i(\mathbb{E})$ of the Hodge bundle. 
The class $\kappa_i$ is defined to be the push-forward $\pi_*(K^{i+1})$, where $K=c_1(\omega).$ The class $\psi_i$ is the pull back $\sigma_i^*(K)$ of $K$ along $\sigma_i:\overline{M}_{g,n} \rightarrow \overline{M}_{g,n+1}$, 
where $\sigma_1,...,\sigma_n$ are the natural sections of the map $\pi.$ It is the first Chern class of the bundle on the moduli space whose fiber at the moduli point $[(C;x_1,...,x_n)]$ 
is the cotangent space to $C$ at the $i^{th}$ marking.

The class $\lambda_g$ vanishes when restricted to the complement $\Delta_{irr}.$ This gives rise to an evaluation $\epsilon$ on $A^*(M_{g,n}^{ct}):$ 
$$\xi \mapsto \epsilon( \xi)=\int_{\overline{M}_{g,n}} \xi.\lambda_g.$$ 

The non-triviality of the $\epsilon$ evaluation is proven by explicit integral computations. The following formula for $\lambda_g$ integrals is proven in [FP2]:

$$\int_{\overline{M}_{g,n}} \psi_1^{\alpha_1}...\psi_n^{\alpha_n} \lambda_g= \binom{2g-3+n}{\alpha_1,...,\alpha_n} \int_{\overline{M}_{g,1}} \psi_1^{2g-2} \lambda_g.$$

The integrals on the right side are evaluated in terms of the Bernoulli numbers: $$\int_{\overline{M}_{g,1}} \psi_1^{2g-2} \lambda_g=\frac{2^{2g-1}-1}{2^{2g-1}}\frac{|B_{2g}|}{(2g)!}.$$

This proves the non-triviality of the evaluation since $B_{2g}$ doesn't vanish. 

It is proven in [GV] that $R^*(M_{g,n}^{ct})$ vanishes in degrees $> 2g-3+n$ and is 1-dimensional in degree $2g-3+n.$ It was speculated in [FP1] that 
$R^*(M_g^{ct})$ is a Gorenstein algebra with socle in codimension $2g-3.$ The following conjecture is stated in [F2]:

\begin{conjecture}\label{ct}
$R^*(M_{g,n}^{ct})$ is Gorenstein with socle in degree $2g-3+n.$ 
\end{conjecture}

A compactly supported version of the tautological algebra is defined in [HL]. The algebra $R^*_c(M_{g,n})$ is defined to be the set of elements in 
$R^*(\overline{M}_{g,n})$ that restrict trivially to the Deligne-Mumford boundary. This is a graded ideal in $R^*(\overline{M}_{g,n})$ and the intersection product defines a map 
$$R^*(M_{g,n}) \times R^{*}_c(M_{g,n}) \rightarrow R^{*}_c(M_{g,n})$$ that makes $R^*_c(M_{g,n})$ a $R^*(M_{g,n})$-module. 
In [HL] they formulated the following conjecture for the case $n=0$:

\begin{conjecture}\label{hl}
(A) The intersection pairings $$R^k(M_g) \times R^{3g-3-k}_c(M_g) \rightarrow R_c^{3g-3}(M_g) \cong \mathbb{Q}$$ are perfect for $k \geq 0.$

(B) In addition to (A), $R^*_c(M_g)$ is a free $R^*(M_g)$-module of rank one. 
\end{conjecture}

In a similar fashion one defines $R^*_c(M_{g,n}^{ct})$ as the set of elements in $R^*(\overline{M}_{g,n})$ 
that pull back to zero via the standard map $\overline{M}_{g-1,n+2} \rightarrow \overline{M}_{g,n}$ onto $\Delta_{irr}.$ The analogue of the conjectures above 
for the spaces $M_{g,n}^{ct}$ instead of $M_g$ and its relation with the conjecture \ref{ct} is discussed in [F3].  First consider the analogue of the conjectures \ref{hl} as follows:

\begin{conjecture} \label{HL}
(C) The intersection pairings 
$$R^k(M_{g,n}^{ct}) \times R^{3g-3+n-k}_c(M_{g,n}^{ct}) \rightarrow R_c^{3g-3+n}(M_{g,n}^{ct}) \cong \mathbb{Q}$$ are perfect for $k \geq 0.$

(D) In addition to C, $R^*_c(M_{g,n}^{ct})$ is a free $R^*(M_{g,n}^{ct})$-module of rank one. 
\end{conjecture}

In [F3] it is proven that for a given $(g,n),$ the statement D in \ref{HL} follows if the statements \ref{bar} and \ref{ct} hold.  
On the other hand, for such $(g,n)$ the statements \ref{ct} and C in \ref{HL} follow from D in \ref{HL}.  
It is also proven that a counterexample to the conjecture \ref{bar} leads to a disproof of the conjecture C in \ref{HL}.  

In this note we consider the case $g=1$ and prove that the conjecture \ref{ct} is true in this case. 

\section{\bf The space $\overline{U}_{n-1}$}

Let $C$ be a fixed elliptic curve and choose a point $O \in C$ as its origin. For a given natural number $n \in \mathbb{N},$ the space $U_{n-1}$ is defined to be the open subset 
$$\{(x_1,...,x_{n-1}) \in C^{n-1}: x_i \neq O \ \mathrm{and} \ x_i \neq x_j \ \mathrm{for} \ i \neq j  \}$$ of $C^{n-1}.$ The projection $\pi:U_{n-1} \times C \rightarrow U_{n-1}$ admits $n$ 
disjoint sections with smooth fibers and defines a map $$F:U_{n-1} \rightarrow M_{1,n},$$ where $M_{1,n}$ denotes the moduli space of smooth $n$-pointed curves of genus one. 
The map $F$ sends the point $P=(x_1,...,x_{n-1})$ of $U_{n-1}$ to the class of the pointed curve $(C;x_1,...,x_{n-1},O).$

For a subset $I$ of $\{1,...,n\},$ let $X_I \subset C^{n-1}$ be the $|I|$-dimensional subvariety defined by 
$$\left\{ \begin{array}{ll}
x_i=x_j \qquad \mathrm{for} \ i,j \in \{1,...,n\}-I & \mathrm{if}  \ n \in I \\
\\
x_i=O  \qquad \mathrm{for} \ i \in \{1,...,n-1\}-I & \mathrm{if}  \ n \notin I.  \\
\end{array} \right.$$

The space $\overline{U}_{n-1}$ is constructed from $C^{n-1}$ by the following sequence of blow-ups:

At step zero blow-up $C^{n-1}$ at the point $X_0,$ and at the $k^{th}$ step, for $1 \leq k \leq n-3,$ 
blow-up the space obtained in the previous step along the regularly embedded union of the proper transforms of the subvarieties $X_I,$ where $|I|=k.$ 

The space $\overline{U}_{n-1}$ contains $U_{n-1}$ as an open dense subset. There exists a family of stable curves of genus one of compact type over $\overline{U}_{n-1},$ 
whose total space is isomorphic to $\overline{U}_n.$ The resulting family is denoted by $\pi:\overline{U}_n \rightarrow \overline{U}_{n-1}$ by abuse of notation. 
Since $\pi^{-1}(U_{n-1})$ is isomorphic to the product $U_{n-1}\times C,$ on which $\pi$ is projection onto the first factor, and this coincides with the former definition of $\pi$ 
given above, there is no danger of confusion. The map $\pi$ admits $n$ disjoint sections in the smooth locus of the fibers, and defines a morphism 
$$F: \overline{U}_{n-1} \rightarrow M_{1,n}^{ct}.$$ The morphism $F$ sends a geometric point $P \in \overline{U}_{n-1}$ to the moduli point of the pointed curve 
$(\pi^{-1}(P);x_1,...,x_n),$ where the $x_i$'s are the $n$ distinct points on the fiber $\pi^{-1}(P)$ obtained by intersecting the fiber $\pi^{-1}(P)$ with the $n$ disjoint sections of $\pi.$

\section{\bf The Chow ring $A^*(\overline{U}_{n-1})$}

In this section we recall some facts about the intersection ring of the blow-up $\widetilde{Y}$ of the smooth variety $Y$ along a smooth irreducible subvariety $Z$ 
from [FM]. When the restriction map from $A^*(Y)$ to $A^*(Z)$ is surjective, S. Keel has shown in [K] that the computations become simpler. 
We denote the kernel of the restriction map by $J_{Z/Y}$ so that $$A^*(Z)=\frac{A^*(Y)}{J_{Z/Y}}.$$ Define a Chern polynomial for $Z \subset Y,$ denoted by 
$P_{Z/Y}(t),$ to be a polynomial $$P_{Z/Y}(t)=t^d+a_1t^{d-1}+...+a_{d-1}t+a_d \in A^*(Y)[t],$$ where $d$ is the codimension of $Z$ in $Y$ 
and $a_i \in A^i(Y)$ is a class whose restriction in $A^i(Z)$ is $c_i(N_{Z/Y}),$ where $N_{Z/Y}$ is the normal bundle of $Z$ in $Y.$ We also require that $a_d=[Z],$ 
while the other classes $a_i,$ for $0<i<d,$ are determined only modulo $J_{Z/Y}.$

Let us verify the surjectivity of the restriction map from $A^*(Y)$ to $A^*(Z)$ in our case, when $Y=C^{n-1}$ and $Z=X_I$, for a subset $I$ of the set $\{1,...,n\}.$ 
First assume that $n$ doesn't belong to the set $I.$ Denote by $i_I:X_I \rightarrow C^{n-1}$ the inclusion map and by $\pi:C^{n-1} \rightarrow X_I$ the canonical projection. From the equality $\pi \circ i_I=id_{X_I}$
we conclude that the restriction map $i_I^*$ is surjective. It also follows that the push-forward map $(i_I)_*$ is injective. This will be used in \ref{R}. The case $n \in I$ is treated in a similar manner. In this case there is not a canonical 
projection $\pi:C^{n-1} \rightarrow X_I$, and one has to make a choice.

The following lemma can be used to compute $P_{Z/Y}$ when the subvariety $Z$ is a transversal intersection of divisor classes:

\begin{lemma}
(a) If $Z=D$ is a divisor, then $P_{D/Y}(t)=t+D.$

(b) If $V \subset Y$ and $W \subset Y$ are subvarieties meeting transversally in a variety $Z,$ and $V$ and $W$ have Chern polynomials $P_{V/Y}(t)$ and $P_{W/Y}(t),$ then $Z$ 
has a Chern polynomial $$P_{Z/Y}(t)=P_{V/Y}(t).P_{W/Y}(t).$$ In addition the restriction from $A^*(Y)[t]$ to $A^*(V)[t]$ maps $P_{W/Y}(t)$ to a Chern polynomial $P_{Z/V}(t)$ for $Z \subset V.$
\end{lemma}
\begin{proof}
This is Lemma 5.1 in [FM].
\end{proof}

We identify $A^*(Y)$ as a subring of $A^*(\widetilde{Y})$ by means of the map $\pi^*:A^*(Y) \rightarrow A^*(\widetilde{Y}),$ 
where $\pi:\widetilde{Y}\rightarrow Y$ is the birational morphism. Let $E \subset \widetilde{Y}$ be the exceptional divisor. The formula of Keel is as follows:

\begin{lemma}\label{K}
With the above assumptions and notations, the Chow ring $A^*(\widetilde{Y})$ is given by $$A^*(\widetilde{Y})=\frac{A^*(Y)[E]}{(J_{Z/Y}.E,P_{Z/Y}(-E))}.$$
\end{lemma}
\begin{proof}
This is Lemma 5.3 in [FM].
\end{proof}

The next lemma relates a Chern polynomial $P_{\widetilde{V}/\widetilde{Y}}(t)$ of the proper transform $\widetilde{V}$ of a subvariety $V \subset Y$ to $P_{V/Y}(t):$

\begin{lemma}\label{proper}
Let $V$ be a subvariety of $Y$ not contained in $Z$ and let $\widetilde{V} \subset \widetilde{Y}$ be its proper transform. Suppose that $P_{V/Y}(t)$ is a Chern polynomial for $V.$

\begin{enumerate}
\item If $V$ 
meets $Z$ transversally, then $P_{V/Y}(t)$ is a Chern polynomial for $\widetilde{V}$ in $\widetilde{Y}.$
\item If $V$ contains $Z,$ then $P_{V/Y}(t-E)$ is a Chern polynomial for $\widetilde{V} \subset \widetilde{Y}.$
\end{enumerate}
\end{lemma}
\begin{proof}
This is Lemma 5.2 in [FM].
\end{proof}

We also need the following lemmas to relate the ideal $J_{\widetilde{V}/\widetilde{Y}}$ to the ideal $J_{V/Y}$ for a subvariety $V$ of $Y:$

\begin{lemma}
Suppose that $V$ is a nonsingular subvariety of $Y$ that intersects $Z$ transversally in an irreducible variety $V \cap Z,$ and that the restriction $A^*(V) \rightarrow A^*(V \cap Z)$ is also surjective. 
Let $\widetilde{V}=Bl_ZV.$ Then $A^*(\widetilde{Y}) \rightarrow A^*(\widetilde{V})$ is surjective, with kernel $J_{V/Y}$ if $V \cap Z$ is not empty, and kernel $(J_{V/Y},E)$ if $V \cap Z$ is empty.
\end{lemma}
\begin{proof}
This is Lemma 5.4 in [FM].
\end{proof}

\begin{lemma}\label{kernel}
Suppose that $Z$ is the transversal intersection of nonsingular subvarieties $V$ and $W$ of $Y,$ and that the restrictions $A^*(Y) \rightarrow A^*(V)$ and $A^*(Y) \rightarrow A^*(W)$ are also surjective, 
let $\widetilde{V}=Bl_Z V.$ Then
\begin{enumerate}
\item $A^*(\widetilde{Y}) \rightarrow A^*(\widetilde{V})$ 
is surjective, with kernel $(J_{V/Y},P_{W/Y}(-E))$;
\item $A^*(\widetilde{Y})\rightarrow A^*(E \cap \widetilde{V})$ 
is surjective, with kernel $(J_{Z/Y},P_{W/Y}(-E)).$
\end{enumerate}
\end{lemma}
\begin{proof}
This is Lemma 5.5 in [FM].
\end{proof}

Using the general results mentioned above we are able to express certain monomials that belong to the Chow ring $A^*(\widetilde{Y})$ in terms of elements in $A^*(Y):$ 

\begin{lemma}\label{num1}
Suppose that $Z$ is the transversal intersection $D_1\cap ... \cap D_r$ of divisor classes $D_1,...,D_r$ on $Y$ and let $f \in A^*(Y)$ be an element of degree $d=\dim(Z).$ 
The following relation holds in $A^*(Bl_Z Y):$ $$f.E^r=(-1)^{r-1}f.Z.$$\end{lemma}
\begin{proof}
Multiply both sides of the equality $$(D_1-E)...(D_r-E)=0$$ by $f.$ For any element $g \in A^*(Y)$ of positive degree, the pull back $i_Z^*(fg)$ of $fg$ along the inclusion 
$i_Z:Z \rightarrow Y$ is zero, which means that the product $fg.E$ is zero as well. This proves the claim.
\end{proof}

We also state the non-vanishing criteria of the product $E_I.E_J$ for a pair of exceptional divisors $E_I$ and $E_J:$

\begin{prop}
Let $I,J \subset \{1,...,n\}$ be subsets satisfying $|I|,|J| \leq n-3.$ The product $E_I.E_J \in A^2(\overline{U}_{n-1})$ is zero unless $I \subseteq J$ or $J \subseteq I$ or $I \cup J=\{1,...,n\}.$
\end{prop}
\begin{proof}
If $I \cup J \neq \{1,...,n\},$ then $X_{I \cap J}$ is equal to the intersection $X_I \cap X_J,$ and it is a proper subset of $X_I$ and $X_J$ both if $I \nsubseteq J$ and $J \nsubseteq I.$ Under the assumption 
$I \nsubseteq J,J \nsubseteq I$ and $I \cup J \neq \{1,...,n\}$, the proper transforms of the subvarieties $X_I$ and $X_J$ become disjoint after blowing up along that of $X_{I \cap J}.$ 
This means that the product $E_I.E_J \in A^2(\overline{U}_{n-1})$ is zero.
\end{proof}

\section{\bf The tautological ring $R^*(C^n)$}
\begin{defn}
Suppose $(C;O)$ is a fixed pointed elliptic curve, and let $n \in \mathbb{N}$ be a natural number. The tautological ring, $R^*(C^n) \subset A^*(C^n)$, is defined to be the $\mathbb{Q}$-subalgebra generated by the following classes:
$$a_i=\{(x_1,...,x_n)\in C^n: x_i=O\}, \qquad d_{j,k}=\{(x_1,...,x_n)\in C^n: x_j=x_k\},$$ where $1 \leq i \leq n$ and $1 \leq j < k \leq n.$ If we define $b_{j,k}:=d_{j,k}-a_j-a_k,$ 
then another set of generators for $R^*(C^n)$ is $\{a_i,b_{j,k}: \ 1 \leq i \leq n \ \mathrm{and} \ \ 1 \leq j < k \leq n\}.$
\end{defn}

\begin{prop}
(A) The space of relations in $R^*(C^n)$ is generated by the following ones: $$a_i^2=0, \qquad a_ib_{i,j}=0,\qquad b_{i,j}b_{i,k}=a_ib_{j,k}, \qquad b_{i,j}b_{k,l}+b_{i,k}b_{j,l}+b_{i,l}b_{j,k}=0,$$
where in each relation the indices are distinct.

(B) For any $0 \leq d \leq n,$ the pairing $R^d(C^n) \times R^{n-d}(C^n) \rightarrow \mathbb{Q}$ is perfect.
\end{prop}
\begin{proof}
We first verify that the relations above hold in $R^2(C^n).$ The relations $a_i^2=a_ib_{i,j}=0$ obviously hold. E. Getzler proved in [G] that the following relation holds in $A^2(\overline{M}_{1,4}):$
\begin{equation} \tag{1}\label{G} 12 \delta_{2,2} - 4 \delta_{2,3} - 2\delta_{2,4} + 6\delta_{3,4}+\delta_{0,3}+\delta_{0,4}-2\delta_{\beta} = 0, \end{equation} where the classes above are defined in [G].

In [P], R. Pandharipande gives a direct construction of Getzler's relation via a rational equivalence in the Chow group $A_2(\overline{M}_{1,4})$.
If we restrict the relation \eqref{G} to the space $M_{1,4}^{ct},$ pull it back along the morphism $F:\overline{U}_3 \rightarrow M_{1,4}^{ct},$ 
and push it down to $C^3$ via the blow-down map, we get the relation \begin{equation}  \tag{2}\label{e1}12(a_1b_{2,3}-b_{1,2}b_{1,3})=0.\end{equation}
Next, we deal with the last relation. Denote by $\pi:M_{1,5}^{ct} \rightarrow M^{ct}_{1,4}$ 
the morphism which forgets the fifth marking. If we pull back the restriction of the relation \eqref{G} by the map $\pi$, 
its pull-back along $F:\overline{U}_4 \rightarrow M_{1,5}^{ct}$ gives a relation whose push-down by the blow-down map to $C^4$ 
becomes \begin{equation}  \tag{3}\label{e2} 12(b_{1,2}b_{3,4}+b_{1,3}b_{2,4}+b_{1,4}b_{2,3})=0.\end{equation}
For more details about the derivation of \eqref{e1} and \eqref{e2}, please see the appendix.

Now, we study the pairing $$R^d(C^n) \times R^{n-d}(C^n) \rightarrow \mathbb{Q}$$ for $0 \leq d \leq n.$ 
From the relations above, we see that the tautological group $R^d(C^n)$ is generated by monomials of the form 
$v=a(v).b(v),$ where $a(v)$ is a product $\prod a_i$ of $a_i$'s for $i \in A_v,$ and $b(v)$ is a product $\prod b_{j,k}$ of $b_{j,k}$'s, for $j,k \in B_v,$ 
such that $A_v$ and $B_v$ are disjoint subsets of the set $\{1,...,n\}$ satisfying $d=|A_v|+\frac{1}{2}|B_v|.$ Under this circumstance, 
the monomial $v$ is said to be standard. To any standard monomial $v$ we associate a dual element $v^* \in R^{n-d}(C^n),$ 
which is defined to be the product of all $a_i$'s, for $i \in \{1,...,n\}-A_v \cup B_v,$ with $b(v).$ The following lemma enables us to study the pairing:

\begin{lemma}
Let $v\in R^d(C^n)$ and $w \in R^{n-d}(C^n)$ be standard monomials. If the product $v.w$ is nonzero, then $B_v=B_w,$ 
and the disjoint union of the sets $A_v,A_w$ and $B_v=B_w$ is equal to the set $\{1,...,n\}.$
\end{lemma}
\begin{proof}
By assumption, we obtain the following inequalities: $$n=(|A_v|+\frac{1}{2}|B_v|)+(|A_w|+\frac{1}{2}|B_w|) \leq |A_v|+|A_w|+|B_v \cup B_w| \leq n.$$
This forces the inequalities to be equalities. The equality $$(|B_v \cup B_w|-|B_v|)+(|B_v \cup B_w|-|B_w|)=0$$ 
implies that $|B_v \cup B_w|=|B_v|=|B_w|,$ which shows that $B_v=B_w.$ The equality $|A_v|+|A_w|+|B_v|=n$ proves the second part of the claim.
\end{proof}
So, after a suitable enumeration of generators for $R^d(C^n),$ the resulting intersection matrix of the pairing between standard 
monomials and their dual consists of square blocks along the main diagonal and the off-diagonal blocks are all zero. To prove that the pairing is perfect we need to study the square blocks on the main diagonal. 
These matrices and their eigenvalues are studied in [HW]. In particular, from their result it follows that the kernel of any such matrix is generated by relations of the form \eqref{e2}:

\begin{lemma}
Let $m \geq 2$ be a natural number and $S$ be the set of all standard monomials $v$ of the form $b_{i_1,j_1}...b_{i_m,j_m}$ in $R^m(C^{2m}).$ The kernel of the intersection matrix $(v.w)$ for $v,w$ in $S$
is generated by expressions of the form$$R_{\{i,j,k,l\}}:=b_{i,j}b_{k,l}+b_{i,k}b_{j,l}+b_{i,l}b_{j,k},$$ where the indices are distinct elements varying over the set $\{1,...,2m\}.$ 
\end{lemma}
\begin{proof}
The intersection matrix $(v.w)$ for $v,w \in S$ in [HW] is denoted by $T_r(x)$ for $r=m$ and $x=-2.$ The $S_{2m}$-module generated by elements of $S$ decomposes into the sum $\oplus_{\lambda} V_{\lambda},$ 
where $\lambda$ varies over all partitions of $2m$ into even parts. For each such $\lambda$ the space $V_{\lambda}$ is an eigenspace of $T_r(x).$ 
The corresponding eigenvalue is zero  when $\lambda \neq 2^m$ and it is $(-1)^m(m+1)!$ when $\lambda = 2^m.$  We identify the space $V_{\lambda}$ 
with a subspace of $R^m(C^{2m})$, defined below, which is generated by expressions of the form $R_{\{i,j,k,l\}}$, for $\lambda \neq 2^m.$

Recall that a tabloid is an equivalence class of numberings of a Young diagram, two being equivalent if corresponding rows contain the same entries. 
The tabloid determined by a numbering $T$ is denoted $\{T\}.$ The space $V_{\lambda}$ is generated by elements of the form 
$$v_T=\sum_{q \in C(T)} \mathrm{sgn}(q)\{q.T\},$$ where $C(T)$ is the subgroup of $S_{2m}$ consisting of permutations preserving the columns of $T.$ 

Note that the sum $R_{\{1,...,2m\}}:=\sum_{v \in S}v$ is a symmetric expression, which is clearly  a linear combination of terms of the form $R_{T},$ 
where $|T|=4.$ This proves the claim when $\lambda=2m$ gives the trivial representation. For other partitions $\lambda$ we use the proven result for the symmetric relations. 
Let $\lambda=(\lambda_1 \geq ... \geq \lambda_r > 0)$ be a partition of $2m.$ For each numbering of a Young diagram $T$ let $T_i$ denote the subset of $\{1,...,2m\}$ 
containing elements of the $i^{th}$ row of $T,$ for $i=1,...,r.$  Consider the product $P_T:=\prod_{i=1}^r R_{T_i},$ where $R_{\{i,j\}}$ 
is defined to be $b_{i,j},$ while the other $R_{T_i}$'s are defined above  when $|T_i| \geq 4.$ Note that $P_T$ doesn't change as $T$ varies in an equivalence class 
$\{T\}$ since $R_{T_i}$'s are symmetric. This means that the assignment $$v_T \rightarrow \sum_{q \in C(T)} \mathrm{sgn}(q)P_{q.T},$$ 
is a well-defined $S_{2m}$-module morphism. This map is non-zero, hence an isomorphism onto its image. The result follows.
\end{proof}

Since the relations of the form \eqref{e2}  hold in the tautological ring $R^*(C^n),$ we conclude that the pairing is perfect. 
This also shows that the relations stated in the proposition generate all relations in the tautological ring.
\end{proof}

\begin{rem}
The tautological ring $R^*(C^n),$ for a smooth curve $C$ of genus $g,$ was defined by C. Faber and R. Pandharipande  (unpublished) as the $\mathbb{Q}$-subalgebra 
of $A^*(C^n)$ generated by the standard classes $K_i$ and $D_{i,j}.$ They show that the image $RH^*(C^n)$ in cohomology is Gorenstein. 
In [GG] M. Green and P. Griffiths have shown that $R^*(C^2)$ is not Gorenstein, for $C$ a generic complex curve of genus $g \geq 4.$  
\end{rem}

\section{\bf The tautological ring $R^*(\overline{U}_{n-1})$}

\begin{defn}
Let $Y$ be a blow-up of $C^{n-1}$ at some step in the construction of $\overline{U}_{n-1},$ and denote by $\mathcal{S}_Y$ the collection of subsets of the set $\{1,...,n\}$ 
corresponding to the involved blow-up centers. The tautological ring $R^*(Y)$ of $Y$ is defined to be the subalgebra of the Chow ring $A^*(Y)$ generated by the tautological classes in $R^*(C^{n-1})$ 
and the classes $E_I,$ for $I$ in $\mathcal{S}_Y.$ In particular, $R^*(\overline{U}_{n-1})$ is generated over $R^*(C^{n-1})$ by all $E_I$'s, where $|I| \leq n-3.$
\end{defn}

\subsection{Relations in $R^*(\overline{U}_{n-1})$}

\begin{itemize}\label{R}
\item
For subsets $I,J \subset \{1,...,n\},$ where $|I|,|J| \leq n-3$, the product $E_I.E_J \in R^2(\overline{U}_{n-1})$ is zero unless $$* \qquad I \subseteq J, \qquad  \mathrm{or} \ J \subseteq I, \qquad \mathrm{or} \ I \cup J=\{1,...,n\} .$$

\item For any subset $I \subset \{1,...,n\},$ where $|I| \leq n-3,$ consider the inclusion $$i_I:X_I \rightarrow C^{n-1}.$$ 
The relations $$\{x.E_I=0 : x \in \ker(i_I^*:R^*(C^{n-1}) \rightarrow R^*(X_I))\}$$ hold. Note that the kernel of $i_I^*$ coincides with the kernel of the map 
$$R^*(C^{n-1}) \rightarrow R^*(C^{n-1})$$ defined by $x \rightarrow x.X_I.$ This follows since $(i_I)_* (i_I)^*(x)=x.X_I$ for $x \in R^*(C^{n-1})$, and $(i_I)_*$ 
is injective in our case.

\item As we saw in the third section, in blowing-up the variety $Y$ along a subvariety $Z \subset Y,$ if the center $Z$ 
can be written as the transversal intersection of the subvarieties $V$ and $W$ of $Y,$ then the class $P_{W/Y}(-E_Z)$ is in the ideal $J_{\widetilde{V}/\widetilde{Y}}.$ 
This means that the product $P_{W/Y}(-E_Z).E_{\widetilde{V}}$ is zero, where $E_{\widetilde{V}}$ is the class of the exceptional divisor of the blow-up along the subvariety $\widetilde{V}.$ 
We get a class of relations of this type by writing the centers of blow-ups introduced in the construction of the space $\overline{U}_{n-1}$ as  transversal intersections in different ways. 
If the subvariety $V$ can be written as a transversal intersection $V_1 \cap ... \cap V_k,$ we obtain the relation $P_{W/Y}(-E_Z).E_{V_1}...E_{V_k}=0.$

\item For each subvariety $Z \subset Y$ with a Chern polynomial $P_{Z/Y}(t),$ there is a relation $$P_{Z/Y}(-E_Z)=0,$$ where $E_Z$ is the class of the exceptional divisor of the blow-up of 
$Y$ along $Z.$ These give another class of relations in $R^*(\overline{U}_{n-1}).$ Note that for each subset $I$ of $\{1,...,n\},$ a Chern polynomial $P_{X_I/C^{n-1}}$ of the subvariety $X_I$
is in $R^*(C^{n-1})[t].$ This means that a Chern polynomial of its proper transform under later blow-ups belongs to $R^*(\overline{U}_{n-1}).$ It follows from Lemma ~\ref{proper}, 
which relates a Chern polynomial $P_{V/Y}(t)$ of a subvariety $V$ to that of its proper transform $\widetilde{V}.$
\end{itemize}

\begin{ex} Suppose $Y=C^5.$
\begin{itemize}
\item Let $X_0$ be the point $O^5 \in C^5.$ Then $\ker(i^*:R^*(C^{5}) \rightarrow R^*(X_0))$ consists of all elements of positive degree. 
It follows that $a_i.E_0=b_{i,j}.E_0=0$ for all $i$ and $j.$
\item Let $X_1=a_2a_3a_4a_5.$ From $a_1 \cap X_1=X_0$ we get the relation $(a_1-E_0)E_1=0.$ If $X_{1,2,3}=a_4a_5$ and $X_{4,5,6}=d_{1,2}d_{1,3}$, 
then the relation $(a_1-E_0)E_{1,2,3}E_{4,5,6}=0$ follows from the equality $a_1 \cap X_{1,2,3} \cap X_{4,5,6}=X_0.$
\item A Chern polynomial of the subvariety $X_0$ is $\prod_{i=1}^5(a_i+t)$, from which we get the following relation: $$\prod_{i=1}^5(a_i-E_0)=a_1a_2a_3a_4a_5-E_0^5=0.$$
\end{itemize}
\end{ex}
There are few special cases of the relations above which will be useful in the definition of standard monomials and in defining the dual elements:

\begin{lemma}\label{D}
Let $I$ be a subset of the set $\{1,...,n\}$ with at most $n-3$ elements, containing $n.$ For any $i \in I$  and $j,k \in \{1,...,n\}-I,$ the following relations hold in $A^2(\overline{U}_{n-1}):$ 
$$a_j.E_I=a_k.E_I, \qquad b_{j,k}.E_I=-2a_j.E_I, \qquad b_{i,k}.E_I=(\sum_{J \subseteq I-\{i\}}E_J-a_i-a_j).E_I.$$
\end{lemma}
\begin{proof}
Recall that $E_I$ is the exceptional divisor of the blow-up along the proper transform of the subvariety $$X_I=\cap_{j \neq r \in \{1,...,n\}-I}  d_{j,r}=\cap_{k \neq r \in \{1,...,n\}-I}d_{k,r}.$$
The equality $a_j.E_I=a_k.E_I$ follows since $$a_j-a_k \in \ker(i_I^*:R^*(C^{n-1}) \rightarrow R^*(X_I)),$$ where $i_I:X_I \rightarrow C^{n-1}$ denotes the inclusion.

We give another proof as well: from the following equality $$X_I \cap a_j = X_I \cap a_k = X_{I-\{n\}},$$ we obtain the relation$$(a_j-\sum_{J \subseteq I-\{n\}}E_J)E_I=(a_k-\sum_{J \subseteq I-\{n\}}E_J)E_I=0.$$
This gives the first relation after canceling out $(\sum_{J \subseteq I-\{n\}}E_J)E_I$ on both sides.

The second equality results from the definition $b_{j,k}=d_{j,k}-a_j-a_k,$ the relation $d_{j,k}.E_I=0,$ and from the previous case.

To prove the last statement, first note that $b_{i,k}=d_{i,k}-a_i-a_k,$ by definition. From the equality $X_I \cap d_{i,k}=X_{I-\{i\}},$ we get the relation $$(d_{i,k}-\sum_{J \subseteq I-\{i\}}E_J).E_I=0.$$ 
We conclude that $$b_{i,k}.E_I=(\sum_{J \subseteq I-\{i\}}E_J-a_i-a_k).E_I,$$ which proves the last statement, using that $a_j.E_I=a_k.E_I.$
\end{proof}

\subsection{Standard monomials} 
The existence of the relations stated above makes it possible to obtain a smaller set of generators for the tautological ring $R^*(\overline{U}_{n-1}).$ 
Any monomial $v \in R^d(\overline{U}_{n-1})$ can be written as a product $a(v)b(v)E(v),$ where $a(v)$ is a product of $a_i$'s, $b(v)$ is a product of $b_{j,k}$'s, and $E(v)$ 
is a product of exceptional divisors. To simplify the enumeration of the generators for $R^*(\overline{U}_{n-1}),$ we introduce the directed graph associated to a monomial:

\begin{defn}\label{Graph}
Let $v=a(v)b(v)E_{I_1}^{i_1}...E_{I_m}^{i_m}$, where  $i_r \neq 0$ for  $r=1,...,m$ and $I_1 <...< I_m$, be a monomial. 
The directed graph $\mathcal{G}=(V_\mathcal{G},E_\mathcal{G})$ associated to $v$ is defined by the following data:
\begin{itemize}
\item A set $V_\mathcal{G}$ and a one-to-one correspondence between members of $V_{\mathcal{G}}$ and members of the set $\{1,...,m\}.$ Elements of $V_\mathcal{G}$ are called the vertices of $\mathcal{G}.$
\item  A set $E_\mathcal{G} \subset V_\mathcal{G} \times V_\mathcal{G}$ consisting of all pairs $(r,s),$ where $I_s$ is a minimal element of the set $$\{I_i: I_r \subset I_i\}$$ 
with respect to inclusion. Elements of $E_\mathcal{G}$ are called the edges of $\mathcal{G}.$
\end{itemize}

For a vertex $i \in V_\mathcal{G},$ the closure $\overline{i} \subset V_\mathcal{G}$ is defined to be the subset
$$\{r \in V_\mathcal{G}: I_i \subseteq I_r\}$$ of $V_\mathcal{G}.$ The degree $\deg(i)$ of $i$ is defined to be the number of the elements of the set 
$$\{j \in V_\mathcal{G}: (i,j) \in E_\mathcal{G}\}.$$

A vertex $i \in V_\mathcal{G}$ is called a \emph{root} of $\mathcal{G}$ if $I_i$ is minimal with respect to inclusion of sets. 
Maximal vertices of $\mathcal{G}$ are called \emph{external} and all the other vertices will be called \emph{internal}.

In the following, we use the letters $I_1,...,I_m$ to denote the vertices of $\mathcal{G}.$
\end{defn}

\begin{rem}
We can define a directed graph associated to any collection of subsets of the set $\{1,...,n\}$ in a similar way. In general there may be a loop in the resulting graph after forgetting the directions. 
Loops don't occur when the collection consists of only proper subsets $I$ of the set $\{1,...,n\}$ such that for any two distinct members $I,J$ of the collection one of the conditions in $*$ holds. 
Hence, we refer to $\mathcal{G}$ as the associated \emph{forest} of the monomial $v,$ or of the collection $\{I_1,...,I_m\}.$
\end{rem}

The next lemma turns out to be useful in defining the dual element:
\begin{lemma}
Suppose that $I_1,...,I_m$ are proper subsets of the set $\{1,...,n\},$ containing at most $n-3$ elements, with the property that each pair $I_r$ and $I_s$ satisfy $*.$
Let $\mathcal{G}$ be the associated forest. If $n \notin \cap_{r=1}^m I_r,$ then there is a unique root of $\mathcal{G}$ not containing $n.$
\end{lemma}
\begin{proof}
By assumption, there is a root $I_r$ such that $n \notin I_r.$ Uniqueness follows since for any two roots $I_r,I_s$ of $\mathcal{G}$ the equality $I_r \cup I_s = \{1,...,n\}$ holds. This means that their complements 
$I_r^c,I_s^c$ are disjoint. Hence, $n$ belongs to the complement of at most one root.
\end{proof}

\begin{defn}\label{standard}
Let $v$ be as in Definition \ref{Graph}, $\mathcal{G}$ be the associated forrest, and $J_1,...,J_s,$ for some $s \leq m,$ be roots of $\mathcal{G}.$ For each $1 \leq r \leq s$ such that $n \in J_r,$ let $j_r \in \{1,...,n\}-J_r$ 
be the smallest element. The subset $S$ of the set $\{1,...,n-1\}$ is defined as follows:
\begin{itemize}
\item
If $n \in \cap_{r=1}^m I_r,$ put $$S:=\{j_1,...,j_s\} \cup (\cap_{r=1}^m I_r - \{n\}),$$
\item
if $n \notin \cap_{r=1}^m I_r,$ let $J_1$ be the unique root of $\mathcal{G}$ not containing $n.$ In this case$$S:=\{j_2,...,j_s\} \cup (\cap_{r=1}^m I_r).$$
\end{itemize}

The monomial $v$ is said to be standard if
\begin{itemize}
\item The monomial $a(v)b(v)\in R^*(C^S)$ is in standard form according to the definition given in the forth section.
\item For each $r$ we have that $$i_r  \leq \min(n-2-|I_r|, |\cap_{I_r \subset I_s} I_s|-|I_r|+\deg(I_r)-2).$$
\end{itemize}
\end{defn}

To prove that the standard monomials generate the tautological ring, we define an ordering on the \emph{polynomial ring} $$R:=\mathbb{Q}[a_i,b_{j,k},E_I: 1 \leq i \leq n-1,1 \leq j<k \leq n-1,I \subset \{1,...,n\}, \ \mathrm{where} \ |I|\leq n-3].$$

\begin{defn}\label{<}
Let $I,J \subset \{1,...,n\},$ we say that $I<J$ if
\begin{itemize}
\item $|I|<|J|$

\item or, if $|I|=|J|$ and the smallest element in $I-I \cap J$ is less than the smallest element of $J-I\cap J.$
\end{itemize}

Put an arbitrary total order on monomials in $$\mathbb{Q}[a_i,b_{j,k}:1 \leq i \leq n-1,1 \leq j < k \leq n-1].$$

Suppose $v_1,v_2 \in R$ are monomials. We say that $v_1 < v_2$ if we can write them as $$v_1=a(v_1)b(v_1)\prod_{r=1}^{r_0}E_{I_r}^{i_r}.E, \qquad \ \mathrm{and} \ v_2=a(v_2)b(v_2)\prod_{r=1}^{r_0}E_{I_r}^{j_r}.E ,$$ 
where $E=\prod_{r=r_0+1}^m E_{I_r}^{i_r}$, for $I_1 < ... <I_m,$ and $i_{r_0}<j_{r_0};$ 

or, if $r_0=0$ and $a(v_1)b(v_1)<a(v_2)b(v_2).$ 

Furthermore, we say that  $v_1 \ll v_2,$ if for any factor $E_I$ of $v_2$ we have that $v_1 <E_I.$ Note that $v_1 \ll v_2$ implies that $v_1 <v_2.$
\end{defn}

\begin{prop}
The tautological ring $R^*(\overline{U}_{n-1})$ of $\overline{U}_{n-1}$ is generated by standard monomials.
\end{prop}
\begin{proof}
Let $v$ be a monomial given as in Definition \ref{Graph}. We have seen that any monomial in $R^d(C^{n-1})$ can be written in standard form for $0 \leq d \leq n-1.$ 
By Lemma \ref{D}, we may assume that $a(v)b(v)$ is an element of the tautological ring $R^*(C^S)$ of $C^S,$ where $S$ is defined according to the Definition \ref{standard}. 
The statement is proven using induction and from the following observations:

\begin{itemize}
\item From the last class of relations in \ref{R}, for any subset $I_r$ of the set $\{1,...,n\},$ where $|I_r|\leq n-3,$ we can write $E_{I_r}^{n-1-|I_r|}$ 
as a sum of elements which are strictly less than it.
\item Let $\{J_1,...,J_s\}$ be the set of minimal elements of the set $$\{I_i: I_r \subset I_i, \  \mathrm{where} \ 1 \leq i \leq m \}$$ 
From the third class of relations in \ref{R}, the monomial $E_{I_r}^{j} \prod_{i=1}^s E_{J_i}$ 
can be written as a sum of terms which are strictly less than it, where $j=|\cap_{i=1}^s J_i|-|I_r|+s-1.$
\end{itemize}
\end{proof}

\subsection{Definition of the dual element}

Now suppose $v\in R^d(\overline{U}_{n-1})$ is an element of degree $d$ written in standard form. Below, we define the element $v^*,$ which is an element of $R^{n-1-d}(\overline{U}_{n-1}).$ 
As we will see, the property $v^{**}=v$ holds. This shows that there is a one-to-one correspondence between standard monomials in degree $d$ and $n-1-d.$

\begin{defn}\label{dual}
Suppose $v=a(v)b(v)E(v)$ is a standard monomial, where $a(v)b(v)$ is in the tautological ring $R^*(C^{n-1})$ of $C^{n-1},$ and $$E(v)=\prod_{r=1}^m E_{I_r}^{i_r},$$ 
where $i_r \neq 0$ for $r=1,...,m,$ and $I_1 < ...<I_m.$ Let $\mathcal{G}$ be the associated forest, and $J_1,...,J_s,$ for some $s \leq m,$ be the roots of $\mathcal{G}.$ For each $1 \leq r \leq s$ such that $n \in J_r,$ 
let $j_r \in \{1,...,n\}-J_r$ be the smallest element. The subset $T$ of the set $\{1,...,n-1\}$ is defined as follows:
\begin{itemize}
\item
If $n \in \cap_{r=1}^m I_r$ put $$T:=\{j_1,...,j_s\} \cup (\cap_{r=1}^m I_r) - (A_v\cup B_v \cup \{n\}),$$
\item
if $n \notin \cap_{r=1}^m I_r,$ let $J_1$ be the unique root of $\mathcal{G}$ not containing $n.$ In this case $$T:=\{j_2,...,j_s\} \cup (\cap_{r=1}^m I_r) - (A_v\cup B_v).$$
\end{itemize}

For each $1 \leq r \leq m,$ define $j_r$ to be $$ \left\{ \begin{array}{ll}
|\cap_{I_r \subset I_s}I_s|-|I_r|+\deg(I_r)-1-i_r & I_r \ \mathrm{is \ an \ internal \ vertex \ of} \ \mathcal{G} \\ \\
n-1-|I_r|-i_r & I_r \ \mathrm{is \ an \ external \ vertex \ of} \ \mathcal{G}.  \\
\end{array} \right. $$

We define $v^*=a(v^*)b(v^*)E(v^*),$ where  $$a(v^*)=\prod_{i \in T}a_i,  \qquad b(v^*)=b(v),\qquad E(v^*)=\prod_{r=1}^m E_{I_r}^{j_r}.$$
\end{defn}

\begin{rem}
We verify that the dual element $v^*$ is well-defined. We need to show that the integers $j_r$ are non-negative. These integers are indeed positive numbers for $r=1,...,m.$ In the definition of standard monomials we have seen that
$$i_r \leq |\cap_{I_r \subset I_s}I_s|-|I_r|+\deg(I_r)-2.$$This shows that $j_r \geq 1$ when $i_r>0$ and $I_r$ is an internal vertex. The case of external vertices is treated using the inequality $i_r \leq n-2-|I_r|.$
\end{rem}

The following corollary follows from Definition \ref{dual}.

\begin{cor}
Suppose $v \in R^d(\overline{U}_{n-1})$ is a standard monomial and let $v^* \in R^{n-1-d}(\overline{U}_{n-1})$ be its dual. Then $v^*$ is a standard monomial, and furthermore $v^{**}=v.$
\end{cor}

The next lemma will be useful in the proof of the Proposition \ref{tri} and the identity \ref{Number}:

\begin{lemma}\label{deg}
Let $v=a(v)b(v)E(v)$ be as in the Definition \ref{dual}, and $\mathcal{G}$ be the associated forest. For a vertex $i \in V_\mathcal{G}$ corresponding to the subset $I_i$ of the set $\{1,...,n\}$, 
the equality $$\sum_{\overline{i}}(i_r+j_r)=n-1-|I_i|$$ holds. Here $\overline{i}$ is the closure of $i$ in $\mathcal{G}$ defined in the Definition \ref{Graph}.
\end{lemma}
\begin{proof}
It is immediate from the definition of the $j_r$'s above.
\end{proof}

\subsection{The pairing $R^d(\overline{U}_{n-1}) \times R^{n-1-d}(\overline{U}_{n-1})$}
In the previous part, we defined dual elements for standard monomials. Below, we will see that the resulting intersection matrix between the standard 
monomials and their duals consists of square blocks on the main diagonal, whose entries are, up to a sign, intersection numbers in $R^{|S|}(C^{S}),$ for certain sets $S,$ and all blocks under the diagonal are zero. 
To prove the stated properties of the intersection matrix, we introduce a natural filtration \footnote{The definition of this filtration on the tautological ring was formulated after a question of E. Looijenga.} on the tautological ring.

\begin{defn}\label{filter}
Let $v$ be a standard monomial as given in Definition \ref{standard}, and let $J_1,...,J_s$ be roots of the associated forest. Define $p(v)$ to be the degree of the element $$a(v)b(v) \cap_{r=1}^s X_{J_r} \in A^*(C^{n-1}),$$ 
which is the same as the integer $$\deg{a(v)b(v)}+n-|\cap_{r=1}^s J_r|-s.$$ The subspace $F^p R^*(\overline{U}_{n-1})$ 
of the tautological ring is defined to be the $\mathbb{Q}$-vector space generated by standard monomials $v$ satisfying $p(v) \geq p.$
\end{defn}

\begin{prop}\label{fil}
(a) For any integer $p,$ we have that $F^{p+1}R^*(\overline{U}_{n-1}) \subseteq F^pR^*(\overline{U}_{n-1}).$

(b) Let $v \in F^pR^*(\overline{U}_{n-1})$ and $w \in R^d(\overline{U}_{n-1})$ be such that $w \ll v.$ If $p+d \geq n,$ then $v.w$ is zero. In particular, $F^n R^*(\overline{U}_{n-1})$ is zero.
\end{prop}
\begin{proof}
The first statement is immediate from Definition \ref{filter}. Let us prove (b). 
Let $v$ be given as in Definition \ref{standard}. Denote by $Y$ the blow-up of $C^{n-1}$ corresponding to the collection $$\{J \subset \{1,...,n\}: J < J_r \ \mathrm{for} \ 1\leq r \leq s\}$$

Note that the dimension $\dim \cap_{r=1}^s X_{J_r}$ of the transversal intersection $\cap_{r=1}^s X_{J_r}$ is equal to $|\cap_{r=1}^s J_r|+s-1.$ 
The product $$a(v)b(v).\prod_{r=1}^s \widetilde{X}_{J_r}.w \in R^{*}(Y)$$ is zero since its degree, which is $$\deg(a(v)b(v))+n-|\cap_{r=1}^s J_r|-s+d,$$ is at least $n,$ by assumption. 
This means that the product $$a(v)b(v).\prod_{r=1}^s E_{J_r}.w \in R^*(\overline{U}_{n-1}),$$ which is a factor of $v.w,$ is zero as well.
\end{proof}

Using the proven lemma we are able to prove the following vanishing result: 

\begin{prop}\label{tri}
Suppose $v_1,v_2 \in R^d(\overline{U}_{n-1})$ are standard monomials satisfying $E(v_1)<E(v_2).$ Then $v_1.v_2^*=0.$
\end{prop}
\begin{proof}
It is enough to write $v_1.v_2^*$ as a product $v.w,$ for $v,w \in R^*(\overline{U}_{n-1})$ satisfying the properties given in the proposition \ref{fil}. 
To find $v$ and $w,$ let $v_1,v_2$ be given as in Definition ~\ref{<}, and denote by $\{J_1,...,J_s\}$ the set of roots of the graph associated to the monomial 
$E=\prod_{r=r_0+1}^m E_{I_r}^{i_r}.$ By relabeling the roots we may assume that there is an $s_0 \geq 0$ such that $I_{r_0} \subset J_r$ for $1 \leq r \leq s_0,$ 
and the equality $I_{r_0} \cup J_r=\{1,...,n\}$ holds for $s_0 < r \leq s.$ Let $w$ be the product of all monomials in $v_1.v_2^*$ which are strictly less than $E_{I_{r_0}}$ and $v$ 
be the product of the other factors, so that $v_1.v_2^*=v.w.$ Notice that $w \ll v,$ by the definition of $v$ and $w.$ The degree $d$ of $w$ is computed using Lemma \ref{deg}:
$$d=n+j_{r_0}-i_{r_0}+s-s_0-|I_{r_0}^c|-\sum_{r=s_0+1}^s |J_r^c|=j_{r_0}-i_{r_0}+s-s_0+|I_{r_0} \cap J_{s_0+1} \cap ... \cap J_s|$$
$$ \geq s-s_0+1+|I_{r_0} \cap J_{s_0+1} \cap ... \cap J_s|=n-p(v).$$ From $d+p(v)\geq n$ we see that the product $v.w$ is zero.
\end{proof}

To study the blocks on the main diagonal we proceed as follows:  we first prove an identity which reduces the number of exceptional divisors in the product for certain monomials. 

Let $Y$ be a blow-up of $C^{n-1}$ at some step in the construction of $\overline{U}_{n-1}.$ Suppose that $$V_1 \cap ... \cap V_k \cap W=Z$$ is a transversal intersection of tautological classes, 
where $W=D_1 \cap...\cap D_r$ is a transversal intersection of divisors $D_1,...,D_r \in R^1(Y)$, 
and let $f \in R^*(Y)$ be an element of degree $d=dim(Z).$ Denote by $E_Z$ the exceptional divisor of the blow-up 
$Bl_Z Y$ of $Y$ along $Z$ and by $E_{V_1},...,E_{V_k}$ those of the blow-up $\widetilde{Y}$ of $Bl_Z Y$ along the proper transform of the subvarieties $V_1,...,V_k.$ 

It follows from the Lemma \ref{kernel} that $P_{W/Y}(-E_Z) \in J_{\widetilde{V}/\widetilde{Y}}$, for $V=V_1 \cap ... \cap V_k.$  
Using the same argument as in Lemma ~\ref{num1} we observe that the equality $$f.E_Z^rE_{V_1}...E_{V_k}=(-1)^{r-1}f.W.E_{V_1}...E_{V_k}$$ holds in $R^{r+d+k}(\widetilde{Y}).$

If the codimension of the subvariety $V_i$ is $r_i$ and that of $Z$ is $r_0$, then from the proven result in Lemma \ref{num1} one gets the following identity:
\begin{equation}  \tag{4}\label{number}f.E_Z^rE_{V_1}^{r_1}...E_{V_k}^{r_k}=(-1)^{r_0-k-1}f.Z.\end{equation}

Notice that this identity reduces the computation of certain monomials containing the exceptional divisors to one which belongs to the Chow ring $A^*(Y).$ 
We now use this identity to compute the numbers occurring on the main diagonal of the intersection matrix. 

Let $I_1,...,I_m$ be subsets of the set $\{1,...,n\}$ containing at most $n-3$ elements such that for every pair $I_r$ and $I_s$ the property $*$ holds.
Let $\mathcal{G}$ be the associated forest with roots $J_1,...,J_s$, and for each $1 \leq r \leq s$ such that $n \in J_r,$ let $j_r \in \{1,...,n\}-J_r$ 
be the smallest element. If $n \notin \cap_{r=1}^m I_r$, let $J_1$ be the unique root of $\mathcal{G}$ not containing $n.$ 
Define $$E:=\prod_{r=1}^m E_{I_r}^{i_r},$$ where $$i_r= \left\{ \begin{array}{ll} |\cap_{I_r \subset I_s}I_s|-|I_r|+\deg(I_r)-1 & I_r \ \mathrm{is \ an \ internal \ vertex \ of} \ \mathcal{G} \\ \\
n-1-|I_r| & I_r \ \mathrm{is \ an \ external \ vertex \ of} \ \mathcal{G}.  \\
\end{array} \right. $$

Consider an element $f \in \mathbb{Q}[a_i,b_{i,j}: i,j \in S]$, of degree $|\cap_{r=1}^m I_r|+s-1$, where
$$S= \left\{ \begin{array}{ll}
\{j_1,...,j_s\} \cup (\cap_{r=1}^m I_r- \{n\}) & n \in \cap_{r=1}^m I_r \\ \\
\{j_2,...,j_s\} \cup (\cap_{r=1}^m I_r) & n \notin \cap_{r=1}^m I_r.  \\
\end{array} \right.$$
Then from the identity \eqref{number} it follows that 
\begin{equation}  \tag{5}\label{Number}f.E= (-1)^{\varepsilon}. f.\prod_{i \in \{1,...,n-1\}-S}a_i.\end{equation}
where $\varepsilon=n+|\cap_{r=1}^m I_r|+\sum_{i \in V(G)}\deg(i)$, using Lemma \ref{deg}. 

Note that for any $v \in R^d(\overline{U}_{n-1})$, the product $E:=E(v)E(v^*)$ is in the form given above according to the Definition \ref{dual}. 

\begin{thm}
For any $0 \leq d \leq n-1,$ the pairing $$R^d(\overline{U}_{n-1}) \times R^{n-1-d}(\overline{U}_{n-1}) \rightarrow \mathbb{Q}$$ is perfect.
\end{thm}
\begin{proof}
Let $A:=\{v_1,...,v_r\} \subset R^d(\overline{U}_{n-1}),$ where $v_1<...<v_r,$ be the set of standard monomials of degree $d,$ and $\{v_1^*,...,v_r^*\} \subset R^{n-1-d}(\overline{U}_{n-1})$ 
be the set of their duals defined above. For a monomial $$E \in \mathbb{Q}[E_I: I \subset \{1,...,n\} \ \mathrm{and} \ |I| \leq n-3],$$ define $$A_E:=\{v \in A: E(v)=E\}.$$ 
Let $\mathcal{G}$ be the graph associated to the monomial $E,$ and define $S$ as in Definition ~\ref{standard}. For $v_i,v_j \in A_E$ the number $$v_i.v_j^* \in R^{n-1}(\overline{U}_{n-1})=\mathbb{Q}$$ 
is the same as $$(-1)^{\varepsilon}a(v_i)b(v_i).a(v_j^*)b(v_j^*) \in R^{S}(C^{S})=\mathbb{Q},$$ by the identity \eqref{Number}, where $\varepsilon=n+|\cap_{r=1}^m I_r|+\sum_{i \in V(\mathcal{G})}\deg(i).$ 

This means that the intersection matrices $(v_i.v_j^*)$ and $(a(v_i)b(v_i).a(v_j^*)b(v_j^*)),$ for $v_i,v_j$ in the set $A_E,$ are the same up to a sign after the identifications above. 
From the study of the tautological ring $R^*(C^{S})$ of $C^{S},$ we know that the kernel of the matrix above is generated by relations in $R^*(C^{S}).$ 
After choosing a basis for $R^{d-\deg(E)}(C^{S}),$ the resulting matrix is invertible. It means that the intersection matrix of the pairing $$R^d(\overline{U}_{n-1}) \times R^{n-1-d}(\overline{U}_{n-1}) \rightarrow \mathbb{Q}$$ 
with this choice of basis elements for the tautological groups consists of invertible blocks on the main diagonal and zero blocks under the diagonal, hence is invertible. This proves the claim.
\end{proof}

\section{\bf The tautological ring $R^*(M_{1,n}^{ct})$}
In the first part we obtained the morphism $F:\overline{U}_{n-1} \rightarrow M_{1,n}^{ct},$
induced from the family of curves $\pi:\overline{U}_n \rightarrow \overline{U}_{n-1}.$ 
The morphism $F$ induces a ring homomorphism $$F^*: A^*(M_{1,n}^{ct}) \rightarrow  A^*(\overline{U}_{n-1}).$$

For a subset $J \subset \{1,...,n\},$ the pull back $F^*(D_J)$ of the divisor class $D_J$ is a subvariety of $\overline{U}_{n-1}$ for which the fiber $\pi^{-1}(P)$ is a nodal curve of type given by $D_J,$ 
when $P$ is its generic point. It follows that $P$ is a point of the proper transform of the subvariety $X_I,$ where $I:=\{1,...,n\}-J,$ when $|J| \geq 3,$ and belongs to the proper transform of the divisors 
$a_i,d_{j,k}$ if $J=\{i,n\},\{j,k\},$ for $1 \leq i \leq n-1$ and $1 \leq j<k \leq n-1.$ But the proper transform of $X_I$ is equal to $E_I,$ when $|I|\leq n-3,$ and those of the divisors $a_i,d_{j,k}$ are 
$a_i-\sum_{i \notin I \subset \{1,...,n-1\}}E_I$ and $d_{j,k}-\sum_{I \subset \{1,...,n\}-\{j,k\}}E_I$ respectively, for $1 \leq i \leq n-1$ and $1 \leq j<k \leq n-1.$ 
This means that $$F^*(D_{\{i,n\}})=a_i-\sum_{i \notin I \subset \{1,...,n-1\}}E_I \ \mathrm{for} \ 1 \leq i \leq n-1,$$
$$F^*(D_{\{j,k\}})=d_{j,k}-\sum_{I \subset \{1,...,n\}-\{j,k\}}E_I \ \mathrm{for} \  1 \leq j < k \leq n-1,$$ $$F^*(D_J)=E_{\{1,...,n\}-J} \ \mathrm{for} \ |J| \geq 3.$$ From this we see that the pull-back homomorphism $F^*$ 
sends tautological classes to tautological classes, and defines a ring homomorphism $$F^*: R^*(M_{1,n}^{ct}) \rightarrow  R^*(\overline{U}_{n-1}).$$

If we rewrite the expressions above, we get that $$a_i=F^*(\sum_{i ,n \in I}D_I)\ \mathrm{for} \ 1 \leq i \leq n-1, \qquad d_{i,j}=F^*(\sum_{j,k \in I}D_I) \ \mathrm{for} \  1 \leq j < k \leq n-1,$$
$$E_I=F^*(D_{\{1,...,n\}-I})  \ \mathrm{for} \ |I| \leq n-3.$$ This shows that $F^*$ is a surjection. We prove that $F^*$ 
is injective by extending the function $$G:\{a_i,d_{j,k},E_I: 1 \leq i \leq n-1 , 1 \leq j < k \leq n-1,I \subset \{1,...,n\} \ \mathrm{and} \ |I| \leq n-3 \} \rightarrow R^*(M_{1,n}^{ct}),$$
defined by the rule $$G(a_i)=\sum_{i ,n \in I}D_I\ \mathrm{for} \ 1 \leq i \leq n-1, \qquad G(d_{j,k})=\sum_{j,k \in I}D_I \ \mathrm{for} \  1 \leq i \leq n-1 \ \mathrm{and} \ 1 \leq j < k \leq n-1,$$
$$G(E_I)=D_{I^c} \ \mathrm{for} \ |I| \leq n-3$$ to a ring homomorphism $$G:R^*(\overline{U}_{n-1}) \rightarrow R^*(M_{1,n}^{ct}).$$ 
This is done by verifying that all relations between elements $a_i,d_{j,k},E_I$'s on the left hand side hold between classes on the right hand side. 
To simplify the notation, we drop the letter $G$ for tautological classes in $R^*(M_{1,n}^{ct}).$ For instance, we write $a_i=\sum_{i ,n \in I}D_I$ and $E_I=D_{I^c}$ for a subset $I \subset \{1,...,n\}$ with $|I|\leq n-3.$

Let us introduce the following notation: suppose $S$ is a subset of the set $\{1,...,n\}.$ By $M_{1,S}^{ct},$ we mean the moduli space of stable curves of genus one of compact type whose marking set is $S.$
Let $$\pi_S:M_{1,n}^{ct} \rightarrow M_{1,S}^{ct}$$ be the projection which forgets the markings in $\{1,...,n\}-S$ and contracts unstable components.

\begin{itemize}
\item
We first deal with relations among generators of $R^*(C^{n-1}):$ notice that $$a_i=\pi_{\{i,n\}}^*(D_{\{i,n\}}),  \qquad  d_{j,k}=\pi_{\{j,k\}}^*(D_{\{j,k\}}), \qquad b_{j,k}=\pi_{\{j,k,n\}}^*(D_{\{j,k\}}-D_{\{j,n\}}-D_{\{k,n\}}-D_{\{j,k,n\}}).$$
From the relations $D_{\{i,n\}}^2=D_{\{j,k\}}^2=0$ in $R^2(M_{1,\{i,n\}}^{ct})$ and $R^2(M_{1,\{j,k\}}^{ct}),$ we obtain that the relations $a_i^2=d_{j,k}^2=0$ hold in $R^2(M_{1,n}^{ct}),$ for $1\leq i \leq n-1$ and $1 \leq j < k \leq n-1.$
On the other hand, the relation $$(D_{\{i,j\}}-D_{\{i,n\}}-D_{\{j,n\}}-D_{\{i,j,n\}})(D_{\{i,n\}}+D_{\{i,j,n\}})=0 \in R^2(M_{1,\{i,j,n\}}^{ct})$$ says that $a_i.b_{i,j}=0.$ From the relation $d_{j,k}^2=0$ obtained above, we get that $b_{j,k}^2=-2a_ja_k.$
Now suppose that $i,j,k,l$ are distinct elements of the set $\{1,..,n-1\}.$ As we saw in the forth section, the relation ~\eqref{G} in $R^2(M^{ct}_{1,\{i,j,k,n\}})$ can be written as $$12(a_ib_{j,k}-b_{i,j}b_{i,k})=0.$$
The relation $$12(b_{i,j}b_{k,l}+b_{i,k}b_{j,l}+b_{i,l}b_{j,k})=0$$ is the pull-back of the relation above to $R^2(M^{ct}_{1,\{i,j,k,l,n\}})$ along the morphism $$\pi_{\{i,j,k,n\}}:M^{ct}_{1,\{i,j,k,l,n\}} \rightarrow M^{ct}_{1,\{i,j,k,n\}}.$$
This shows that the classes $a_i,b_{j,k} \in R^*(M_{1,n}^{ct})$ satisfy all relations among $a_i,b_{j,k} \in R^*(\overline{U}_{n-1}).$

\item
Note that the following $$D_I.D_J \neq 0 \Rightarrow I \subseteq J \ \mathrm{or} \ J \subseteq I \ \mathrm{or} \ I \cap J =\emptyset$$ is true. 
But this can be written as $$E_I.E_J \neq 0 \Rightarrow I \subseteq J \ \mathrm{or} \ J \subseteq I \ \mathrm{or} \ I \cup J=\{1,...,n\}.$$ 
This proves that the $E_I$'s in $R^*(M_{1,n}^{ct})$ satisfy the first class of relations between $E_I$'s in $R^*(\overline{U}_{n-1})$ obtained above.

\item
For any $I \subset \{1,...,n\}$ with $|I| \leq n-3,$ we found the relations $x.E_I=0$ for $x \in \ker(i_I^*),$ where $i_I:X_I \rightarrow C^{n-1}$ denotes the inclusion. 
If $n \notin I,$ then $\ker(i_I^*)$ is generated by divisor classes $a_i,b_{i,j},$ where $i \in J:=\{1,...,n-1\}-I,$ and $j \in \{1,...,n-1\}$ is different from $i.$
Let us see that $a_i.E_I=0$ in this case:$$a_i.E_I=D_J^2+\sum_{J_0:i,n \in J_0 \subset J}D_{J_0}.D_J+\sum_{J_0:i,n \in J_0,J \subset J_0}D_{J_0}.D_J.$$
But this expression is zero from the following known formula for $\psi$ classes in genus zero and one:

\begin{prop}
(a) The following relation holds in $A^1(\overline{M}_{0,n}):$ $$\psi_i=\sum_{j,k \notin I , i \in I ,|I| \geq 2 }D_I$$ for some fixed distinct $j,k \in \{1,...,n\}-\{i\}.$

(b) The following relation holds in $A^1(M^{ct}_{1,n}):$ $$\psi_i=\sum_{i \in I, |I| \geq 2}D_I.$$
\end{prop}
\begin{proof}
(a) is Proposition 1.6 in [AC]. To prove (b), it is enough to restrict the divisor classes given in Proposition 1.9 of [AC], to the space $M_{1,n}^{ct}.$
\end{proof}

If $i \in \{1,...,n-1\}-I,$ and $j \in \{1,...,n-1\}$ is distinct from $i,$ we saw that $a_i.E_I=0,$ and by the same argument as above we see that 
$a_j.E_I=d_{i,j}.E_J,$ from which it follows that $$b_{i,j}.E_I=(d_{i,j}-a_i-a_j).E_I=0.$$ The case $n \in I$ is proven by the same argument.

\item
We get a relation $P_{W/Y}(-E_Z).E_{V_1}...E_{V_k}=0,$ when the subvariety $Z$ is a transversal intersection $V_1 \cap ... \cap V_k \cap W.$
After possibly relabeling the indices, we can assume that $$Z=X_{I_0}, \qquad V_i=X_{I_i}, \ \mathrm{for} \ 1 \leq i \leq k, \qquad W=\prod_{i=r_0+1}^{r_1}a_{i}.\prod_{j=1}^{k-1}a_{r_j+1},$$
where $r_0 \leq r_1 < ... <r_k <n,$ and $I_0=\{1,...,r_0\},I_i^c=\{r_i+1,...,r_{i+1}\},$ for $1 \leq i <k,$ and $I_k^c=\{r_k+1,...,n\}.$ Let us prove that 
$$P_{W/C^{n-1}}(-\sum_{J \subseteq I_0}E_I)E_{I_1}...E_{I_k}=0 \in R^{r_1-r_0+2k-1}(M_{1,n}^{ct})$$ by showing that any monomial in the expansion of this expression is zero. 
Let $$\prod_{i=r_0+1}^{r_1}E_{J_i}.\prod_{j=1}^{k-1}E_{J_{r_j+1}}.E_{I_1}...E_{I_k},$$ where $$ \ i,n \in J_i^c \ \mathrm{for} \ r_0+1 \leq i \leq r_1, \ \mathrm{and}  \ r_j+1,n \in J_j^c \ \mathrm{for} \ 1 \leq j \leq k-1,$$ be any such monomial.

For $r_0+1 \leq i \leq r_1,$ if the product $E_{J_i}.E_{I_k}$ is non-zero, then $J_i \subseteq I_k-\{i\}.$ For $1 \leq j \leq k-1,$ if the product $E_{J_{r_j+1}}.E_{I_j}$ is non-zero, then $J_{r_j+1} \subseteq I_j.$ On the other hand, since 
$n \notin J_{i},J_{r_j+1}$ for all $i,j,$ the product $E_{J_{i_1}}.E_{J_{i_2}}$ is non-zero only if $J_{i_1} \subseteq J_{i_2}$ or $J_{i_2} \subseteq J_{i_1}.$ It follows that the subsets $J_i$'s are totally ordered by inclusion, 
which means that their intersection is one of them. We conclude that for some $i$ the inclusion $J_i \subseteq I_0$ holds. But this term is excluded from expression above, whence the product is zero.

\item
For any subset $I \subset \{1,...,n\},$ where $|I| \leq n-3,$ we prove that $P_{X_I/C^{n-1}}(-\sum_{J \subseteq I}E_J)$ 
is zero by the same argument as in the previous case, by showing that the monomials occurring in the expansion of the expression above are all zero.
\end{itemize}

The argument above shows that $F^*$ is an isomorphism, and hence, we have the following result:

\begin{thm}
The ring homomorphism $F^*:R^*(M_{1,n}^{ct}) \rightarrow R^*(\overline{U}_{n-1})$ is an isomorphism. In particular, for any $0 \leq d \leq n-1,$ the pairing 
$$R^d(M_{1,n}^{ct}) \times R^{n-1-d}(M_{1,n}^{ct}) \rightarrow \mathbb{Q}$$ is perfect. In other words, $R^*(M_{1,n}^{ct})$ is a Gorenstein ring.
 \end{thm}

\section{\textbf{Appendix: Derivation of the relations \eqref{e1} and \eqref{e2}}}
In this appendix we explain why the relations \eqref{e1} and \eqref{e2} follow from Getzler's relation \eqref{G}. First note that the restriction of the relation \eqref{G} 
to the space $M_{1,4}^{ct}$ becomes \begin{equation} \tag{6}\label{Gc} 12 \delta_{2,2} - 4 \delta_{2,3} - 2\delta_{2,4} + 6\delta_{3,4} = 0. \end{equation}
Then we compute the pull-back of the classes above to the space $\overline{U}_3$ along the morphism $$F : \overline{U}_3 \rightarrow M_{1,4}^{ct}.$$
Recall that
\begin{align*}
\delta_{2,2} &= D_{\{1,2\}} \* D_{\{3,4\}} + D_{\{1,3\}} \*
D_{\{2,4\}} + D_{\{1,4\}} \* D_{\{2,3\}} , \\
\delta_{2,3} &= D_{\{1,2\}} \* D_{\{1,2,3\}} + D_{\{1,2\}}
\* D_{\{1,2,4\}} + D_{\{1,3\}} \* D_{\{1,2,3\}} +
D_{\{1,3\}} \* D_{\{1,3,4\}} \\
& + D_{\{1,4\}} \* D_{\{1,2,4\}} + D_{\{1,4\}} \*
D_{\{1,3,4\}} + D_{\{2,3\}} \* D_{\{1,2,3\}} +
D_{\{2,3\}} \* D_{\{2,3,4\}} \\
& + D_{\{2,4\}} \* D_{\{1,2,4\}} + D_{\{2,4\}} \*
D_{\{2,3,4\}} + D_{\{3,4\}} \* D_{\{1,3,4\}} +
D_{\{3,4\}} \* D_{\{2,3,4\}} , \\
\delta_{2,4} &= D_{\{1,2\}} \* D_{\{1,2,3,4\}} + D_{\{1,3\}}
\* D_{\{1,2,3,4\}} + D_{\{1,4\}} \* D_{\{1,2,3,4\}} \\
& + D_{\{2,3\}} \* D_{\{1,2,3,4\}} + D_{\{2,4\}} \*
D_{\{1,2,3,4\}} + D_{\{3,4\}} \* D_{\{1,2,3,4\}} , \\
\delta_{3,4} &= D_{\{1,2,3\}} \* D_{\{1,2,3,4\}} +
D_{\{1,2,4\}} \* D_{\{1,2,3,4\}} + D_{\{1,3,4\}} \* D_{\{1,2,3,4\}} + D_{\{2,3,4\}} \*
D_{\{1,2,3,4\}} .
\end{align*}
From the argument given in section 6 we see that 
$$F^*(D_I ) = E_{\{1,2,3,4\}-I} \ \mathrm{when} \ |I| = 3,4,$$ 
$$F^*(D_{\{i,4\}}) = a_i- E_0- E_j- E_k ,$$
$$ F^*(D_{\{j,k\}}) = d_{j,k} - E_0 - E_i - E_4,$$ 
for  $1 \leq i \leq 3$ and $j \neq k \in \{1, 2,3\}-\{i\},$ from which we conclude that
$$F^*(\delta_{2,2}) = a_1d_{2,3} + a_2d_{1,3} + a_3d_{1,2} + 3E_0^2,$$ 
$$F^*(\delta_{2,3}) = 3(a_1a_2 + a_1a_3 + a_2a_3 + d_{1,2}d_{1,3}) + 3(4E_0 + E_1 + E_2 + E_3 + E_4)E_0,$$
$$F^*(\delta_{2,4}) = -3(2E_0 + E_1 + E_2 + E_3 + E_4)E_0,$$ 
$$F^*(\delta_{3,4}) = (E_1 + E_2 + E_3 + E_4)E_0.$$
If we substitute the expressions above into the relation \eqref{Gc}, we get that $$ 12(a_1d_{2,3}+a_2d_{1,3}+a_3d_{1,2}-a_1a_2-a_1a_3-a_2a_3-d_{1,2}d_{1,3}) = 12(a_1b_{2,3}-b_{1,2}b_{1,3}) = 0.$$
The push-forward of the relation above via the blow-down map to $C^3$ gives the relation \eqref{e1}.

We next deal with the relation \eqref{e2}. Recall that $\pi:\overline{M}_{1,5}\rightarrow \overline{M}_{1,4}$ forgets the fifth marking. It induces a morphism from $M^{ct}_{1,5}$ to $M^{ct}_{1,4},$ 
which is denoted by the same letter by abuse of notation. 
We study the pull-back of the relation \eqref{Gc} along the morphism $\pi \circ F,$ where $F:\overline{U}_4 \rightarrow M_{1,5}^{ct}$ 
is the morphism defined in the first section. It is easy to see that
$$(\pi \circ F)^*\delta_{2,2} = d_{1,2}d_{3,4} + d_{1,3}d_{2,4} + d_{1,4}d_{2,3} + 3(E_0 + E_5)^2,$$
$$(\pi \circ F)^* \delta_{2,3} =12(E_0+E_5)^2+3 \left( d_{1,2}d_{1,3}+d_{1,2}d_{1,4}+d_{1,3}d_{1,4}+d_{2,3}d_{2,4}+\sum_{i=1}^4 (E_0+E_5)(E_i+E_{i,5})\right),$$
$$(\pi \circ F)^*\delta_{2,4}= -6(E_0 + E_5)^2 -\sum_{i \neq j \in \{1,2,3,4\}}d_{i,j}(E_{i,5} + E_{j,5}) - 3 \sum_{i=1}^4E_0E_i,$$
$$(\pi \circ F)^*\delta_{3,4} =\sum_{i=1}^4 E_5E_{i,5} + E_0(E_i + E_{i,5}).$$
The substitution of the expressions above into relation \eqref{Gc} yields$$12(b_{1,2}b_{3,4}+b_{1,3}b_{2,4}+b_{1,4}b_{2,3}) = 0,$$ whose push-forward to $C^4$ via the blow-down map is \eqref{e2}.

\end{document}